\documentclass[11pt]{amsart}
\usepackage{graphicx}
\usepackage{amssymb}
\usepackage{amsthm,enumitem}

\usepackage{fullpage}

\usepackage{hyperref}
\hypersetup{  pdfborder={0 0 0},
  colorlinks   = true,
  urlcolor     = blue,
  linkcolor    = blue,
  citecolor   = red}

\newcommand{\N}{\mathbb{N}}
\newcommand{\Z}{\mathbb{Z}}
\newcommand{\ovl}{\overline}
\newcommand{\cC}{\mathcal{C}}
\newcommand{\cN}{\mathcal{N}}
\newcommand{\cO}{\mathcal{O}}
\newcommand{\lo}{\leqslant_o}

\newcommand{\Gc}{G_{\widehat{\mathcal{C}}}}

\newtheorem{thm}{Theorem}[section]
\newtheorem{prop}[thm]{Proposition}

\newtheorem{cor}[thm]{Corollary}

\theoremstyle{definition}
\newtheorem{defn}[thm]{Definition}

\theoremstyle{remark}
\newtheorem{ex}[thm]{Example}
\newtheorem{rem}[thm]{Remark}

\begin{document}
\title{On double coset separability and the Wilson-Zalesskii property}

\author{Ashot Minasyan}
\address{CGTA, School of Mathematical Sciences,
University of Southampton, Highfield, Southampton, SO17~1BJ, United Kingdom.}
\email{aminasyan@gmail.com}

\begin{abstract}
A residually finite group $G$ has the Wilson-Zalesskii property if for all finitely generated subgroups $H,K \leqslant G$, one has  $\ovl{H} \cap \ovl{K}=\ovl{H \cap K}$, where the closures are taken in the profinite completion $\widehat G$ of $G$. This property played an important role in several papers, and is usually combined with separability of double cosets. In the present note we show that the Wilson-Zalesskii property is actually enjoyed by every double coset separable group. We also construct an example of a LERF group that is not double coset separable and does not have the Wilson-Zalesskii property.
\end{abstract}


\maketitle

\section{Introduction}
Every residually finite group $G$ has a natural embedding into its profinite completion $\widehat G$, which is a compact topological group. The topology on $\widehat G$ induces the \emph{profinite topology} on $G$. A subset $S \subseteq G$ is said to be \emph{separable} if it is closed in this topology, i.e., $S= \ovl{S} \cap G$, where $\ovl{S}$ denotes the closure of $S$ in $\widehat G$.

Many residual properties of $G$ can be interpreted in terms of the profinite topology or the embedding of $G$ into $\widehat G$. In establishing various such properties it is often useful to have control over the intersections of images of two subgroups $H,K \in \mathcal{S}$ in finite quotients of $G$, where $\mathcal{S}$ is a class of subgroups of $G$ (e.g., $\mathcal{S}$ could consist of all cyclic subgroups, all abelian subgroups or all finitely generated subgroups).
The best one can hope for is that for all $H,K \in \mathcal{S}$ we have
\begin{equation}\label{eq:WZ_cond}
\ovl{H} \cap \ovl{K}=\ovl{H \cap K} \text{ in } \widehat{G}
\end{equation}
(see Remark~\ref{rem:control_of_intersec} and Proposition~\ref{prop:tract_intersec_equiv} below, which explain how this is related to controlling the intersection of the images of $H$ and $K$ in finite quotients of $G$).

Condition \eqref{eq:WZ_cond} played an important role in the papers of Ribes and Zalesskii \cite{R-Z-amalg}, Ribes, Segal and Zalesskii \cite{Rib-Seg-Zal}, Wilson and Zalesskii  \cite{Wil-Zal} and Antol\'{\i}n and Jaikin-Zapirain \cite{A-J}, to mention a few. In all of these papers, this condition was established along with (and after) the double coset separability condition, stating that for all $H,K \in \mathcal{S}$
\begin{equation}\label{eq:dc_sep}
HK \text{ is separable in } G.
\end{equation}

The purpose of this note is to demonstrate that condition \eqref{eq:dc_sep} implies \eqref{eq:WZ_cond}, provided $\mathcal S$ is closed under taking finite index subgroups. More precisely, we prove the following.
\begin{thm} \label{thm:main-spec_case} Let $H,K$ be subgroups of a residually finite group $G$. Then the following are equivalent:
\begin{itemize}
  \item[(a)] the double coset $HK$ is separable in $G$ and $\ovl{H} \cap \ovl{K}=\ovl{H \cap K}$ in  $\widehat{G}$;
    \item[(b)] for every finite index subgroup $L \leqslant_f G$, with $H \cap K \subseteq L$, the double coset $(H \cap L)K$ is separable in $G$.
\end{itemize}
\end{thm}

The above theorem follows from Proposition~\ref{prop:tract_intersec_equiv} below, which restates condition \eqref{eq:WZ_cond} in terms of finite index subgroups of the group $G$, and Proposition~\ref{prop:main}, which characterises this restatement in terms of double cosets. Both of these propositions are stated in the general situation of a pro-$\cC$ topology, where $\cC$ is a formation of finite groups. In particular, analogues of Theorem~\ref{thm:main-spec_case} are also true for the pro-$p$ topology, the pro-soluble topology, etc.

Following \cite{A-J}, we say that a group $G$ has the \emph{Wilson-Zalesskii property} if \eqref{eq:WZ_cond} holds for arbitrary finitely generated subgroups $H, K \leqslant G$. This property is named after Wilson and Zalesskii, who established it in the case of finitely generated virtually free groups in \cite{Wil-Zal}. We will call a group $G$ \emph{double coset separable} if \eqref{eq:dc_sep} holds for all finitely generated subgroups $H,K \leqslant G$.

\begin{cor}\label{cor:dc->WZ} Every double coset separable group satisfies the Wilson-Zalesskii property.
\end{cor}

Note that for (virtually) free groups the double coset separability was first proved by Gitik and Rips \cite{Gitik-Rips}. This was extended by Niblo \cite{Niblo} to finitely generated Fuchsian groups and fundamental groups of Seifert-fibred $3$-manifolds. In \cite{M-M} the author and Mineh showed that all finitely generated Kleinian groups and limit groups are double coset separable.
Hence, by Corollary~\ref{cor:dc->WZ}, such groups have the Wilson-Zalesskii property. For limit groups this answers a question of Antol\'{\i}n and Jaikin-Zapirain from \cite[Subsection~2.2]{A-J}.

More generally, separability of double cosets of ``convex'' subgroups is known in many non-positively curved groups (see \cite{Min-GFERF,Groves-Manning,Shepherd-imitator,M-M}). By combining these results with Theorem~\ref{thm:main-spec_case} we gain control over the intersection of such subgroups in finite quotients. Our last corollary describes one such application.

\begin{cor} \label{cor:QCERF_rh_gps} Let $G$ be a finitely generated group hyperbolic relative to a family of double coset separable subgroups. If every finitely generated relatively quasiconvex subgroup is separable in $G$ then  any two finitely generated relatively quasiconvex subgroups $H,K \leqslant G$ satisfy \eqref{eq:WZ_cond}.
\end{cor}

\begin{proof} Let $G$ be a group from the statement. By \cite[Corollary~1.4]{M-M}, the product of two finitely generated relatively quasiconvex subgroups is separable in $G$. Since a finite index subgroup of a relatively quasiconvex subgroup is again relatively quasiconvex \cite[Lemma~5.22]{M-M}, the claim of the corollary follows from Theorem~\ref{thm:main-spec_case}.
\end{proof}

We finish this note by constructing, in Section~\ref{sec:example}, an example of a finitely presented LERF group which is not double coset separable and does not have the Wilson-Zalesskii property.

\subsection*{Acknowledgements} I am grateful to Pavel Zalesskii for fruitful discussions and for drawing my attention to the paper \cite{A-J}, which motivated this note.

\section{A restatement of condition \eqref{eq:WZ_cond}}
Let us fix a formation $\cC$ of finite groups; in other words, $\cC$ is a non-empty class of finite groups which is closed under taking quotients and subdirect products (see \cite[Section~2.1]{Rib-Zal}).

\subsection{Pro-$\cC$ topology and completion}
In this subsection we summarise basic definitions and properties of pro-$\cC$ topology and pro-$\cC$ completions. We refer the reader to  \cite[Sections~3.1, 3.2]{Rib-Zal} for a detailed exposition.

Given a group $G$, we define the \emph{pro-$\cC$ topology} on $G$ by taking the family of normal subgroups $\cN_\cC(G)=\{N \lhd G \mid G/N \in \cC\}$ as a basis of open neighborhoods of the identity element.
A subset $A \subseteq G$ will be called \emph{$\cC$-open} if it is open in the pro-$\cC$ topology on $G$. \emph{$\cC$-closed} and \emph{$\cC$-clopen} subsets are defined similarly.
We will write $H \lo G$ and $N \lhd_o G$ to indicate that $H$ is an open subgroup of $G$ and $N$ is an open normal subgroup of $G$ in the pro-$\cC$ topology.  Note that a subgroup $H \leqslant G$ is $\cC$-open if and only if it contains a $\cC$-open normal subgroup; and $N \lhd G$ is $\cC$-open if and only if $G/N \in \cC$. If $H \lo G$ and $X \subseteq G$ then $XH$ and $G \setminus XH$ are both open as unions of cosets modulo $H$, thus $XH$ is a $\cC$-clopen subset of $G$.

We will use $\Gc$ to denote the \emph{pro-$\cC$ completion} of a group $G$. Equipped with its pro-$\cC$ topology, $\Gc$  is a profinite group; in particular, it is compact.
The natural homomorphism $G \to \Gc$ has dense image. This homomorphism is injective if and only if $G$ is \emph{residually-$\cC$}, i.e., $\bigcap_{N \in \cN_\cC(G)}N=\{1\}$.

\subsection{Tractable intersections}
\begin{defn}
Let $G$ be a group and let $H,K \leqslant G$ be two subgroups. We will say that the intersection $H \cap K$ is \emph{pro-$\cC$ tractable} in $G$ if
for every $M \lhd_o G$ there exists $N \lhd_o G$ such that $N \subseteq M$ and
\begin{equation} \label{eq:tract}
       HN \cap KN \subseteq (H \cap K) M ~\text{ in } G.
\end{equation}
\end{defn}

\begin{rem}\label{rem:control_of_intersec}
Note that condition \eqref{eq:tract} can be restated as $\phi(H) \cap \phi(K) \subseteq \phi(H \cap K) \phi(M)$ in the finite quotient $G/N \in \cC$, where $\phi:G \to G/N$ denotes the natural homomorphism.
\end{rem}

\begin{rem} The following observation will be used throughout this note without further justification. If $A,B$ are subsets of a group $G$ and $H' \leqslant H \leqslant G$ are subgroups then
\[AH' \cap BH=(A \cap BH)H' ~\text{ and }~ H'A \cap HB=H'(A \cap HB).\]
\end{rem}

\begin{prop}\label{prop:tract_intersec_equiv} For subgroups $H, K$ of a residually-$\cC$ group $G$ the following are equivalent:
\begin{itemize}
    \item[(i)] the intersection $H \cap K$ is pro-$\cC$ tractable in $G$;
    \item[(ii)] $\ovl{H} \cap \ovl{K}= \ovl{H \cap K}$ in $\Gc$, where $\ovl{H}$ denotes the closure of $H$ in the pro-$\cC$ completion $\Gc$.
\end{itemize}
\end{prop}

\begin{proof}
Since $G$ is residually-$\cC$, we will treat it as a subgroup of $\Gc$. Note that for an arbitrary $L \lhd_o G$ its closure $\ovl{L}$ is a clopen subgroup of $\Gc$ and $\ovl{L} \cap G=L$, so that $\Gc/\ovl{L}=G/L$ (see \cite[Proposition~3.2.2]{Rib-Zal}). Given any $M \lhd_o G$, let $\cN_\cC(M,G)=\{N \lhd_o G \mid N \subseteq M\}$ and observe that $\{\ovl{N} \mid N \in \cN_\cC(M,G)\}$ is a basis of open neighborhoods of the identity element in $\Gc$.

Let us start with showing that (i) implies (ii). Assuming (i), we know that for every $M \lhd_o G$ there exists $N \in \cN_\cC(M,G)$ such that \eqref{eq:tract} holds. After taking closures of both sides we obtain
\begin{equation} \label{eq:first_inclusion}\ovl{HN \cap KN} \subseteq \ovl{(H \cap K)M} ~\text{ in } \Gc.
\end{equation}
Note that $HN,KN \lo G$, so, by \cite[Proposition~3.2.2]{Rib-Zal}, $\ovl{HN \cap KN}=\ovl{HN} \cap \ovl{KN}$. Clearly $\ovl{H} \cap \ovl {K} \subseteq \ovl{HN} \cap \ovl{KN}$ and
$\ovl{(H \cap K)M}=(H \cap K)\ovl{M}$, because $\ovl{M}$ is a clopen subgroup of $\Gc$.
Hence, in view of \eqref{eq:first_inclusion}, we obtain
\begin{equation} \label{eq:second_inclusion}
 \ovl{H} \cap \ovl{K} \subseteq (H \cap K) \ovl{M} ~\text{ in } \Gc, \text{ for every } M\lhd_o G.
\end{equation}
It is easy to see that $\ovl{H \cap K}=\bigcap_{M \lhd_o G} (H \cap K)\ovl{M}$, because $\cN_{\cC}(\Gc)=\{\ovl{L} \mid L \in \cN_{\cC}(G) \}$. Therefore \eqref{eq:second_inclusion} implies that  $\ovl{H} \cap \ovl{K} \subseteq \ovl{H \cap K}$. The opposite inclusion is obvious, so (ii) has been established.

We will now prove that (ii) implies (i) (in the case of profinite topology this was done in \cite[Corollary~10.4]{A-J}). Suppose that (ii) holds and $M \lhd_o G$ is arbitrary.
If (i) is false, then for every $N \in \cN_\cC(M,G)$, we have
\[(HN \cap KN) \setminus (H \cap K) M\neq \emptyset \text{ in } G, \]
hence
\begin{equation} \label{eq:ovls}
  (H \ovl{N} \cap K \ovl{N}) \setminus (H \cap K) \ovl{M}\neq \emptyset \text{ in } \Gc, \text{ for all } N \in \cN_\cC(M,G),
\end{equation}
where we used the fact that $(H \cap K) \ovl{M} \cap G=(H \cap K) (\ovl{M} \cap G)= (H \cap K)M$.

The family $ \{(H \ovl{N} \cap K \ovl{N}) \setminus (H \cap K) \ovl{M} \mid N \in \cN_\cC(M,G)\}$ consists of clopen sets in $\Gc$ and has the finite intersection property by \eqref{eq:ovls} (because  the intersection of finitely subgroups from $\cN_\cC(M,G)$ is again in $\cN_\cC(M,G)$). Compactness of $\Gc$ now implies that
\begin{equation}\label{eq:big_intersec}
\bigcap_{N \in \cN_\cC(M,G)} (H \ovl{N} \cap K \ovl{N}) \setminus (H \cap K) \ovl{M} \neq \emptyset.
\end{equation}
Since $\bigcap_{N \in \cN_\cC(M,G)} H \ovl{N}=\ovl{H}$, $\bigcap_{N \in \cN_\cC(M,G)} K \ovl{N}=\ovl{K}$ and $\ovl{H \cap K} \subseteq (H \cap K) \ovl{M}$, \eqref{eq:big_intersec} demonstrates
that $(\ovl{H} \cap \ovl{K}) \setminus \ovl{H \cap K} \neq \emptyset$, contradicting (ii). Thus we have proved that (ii) implies (i).
\end{proof}

\section{Characterising tractableness of intersections using double cosets}
As before we will work with a fixed formation of finite groups $\cC$.
For a subgroup $H$ of a group $G$ the pro-$\cC$-topology on $G$ induces a topology on $H$ (which may, in general, be different from the pro-$\cC$ topology of $H$). We will use $\cO_\cC(H,G)$ to denote the open subgroups of $H$ in this induced topology. In other words,
\[\cO_\cC(H,G)=\{H \cap L \mid L \lo G\}.\]
Note that for every $H ' \in \cO_\cC(H,G)$, the index $|H:H'|$ is finite because any $L \lo G$ has finite index in $G$.

\begin{prop} \label{prop:main}
Let $G$ be a group with subgroups $H,K$. Then the following are equivalent:

\begin{itemize}
    \item[(i)] the double coset $HK$ is $\cC$-closed  and the intersection $H \cap K$ is pro-$\cC$ tractable in $G$;
    \item[(ii)] for every $H' \in \cO_\cC(H,G)$, with $H \cap K \subseteq H'$, the double coset $H' K$ is $\cC$-closed in $G$.
    \end{itemize}
\end{prop}

\begin{proof} Let us start with showing that (i) implies (ii). So, assume that (i) is true. Consider any $H' \in \cO_\cC(H,G)$, containing $H \cap K$. Then $H'=H \cap L$, for some $L \lo G$, with $H \cap K \subseteq L$. Let $M \lhd_o G$ denote the normal core of $L$ (it is $\cC$-open by \cite[Lemma~3.1.2]{Rib-Zal}). Since  $H\cap K$ is pro-$\cC$ tractable, there exists $N \lhd_o G$ such that $N \subseteq M$ and
\[HN \cap KN \subseteq (H \cap K)M \subseteq L.\]
Since $NK=KN$, as $N \lhd G$, we can conclude that
\begin{equation*}\label{eq:H_cap_KN_in_H'}
H \cap NK= H \cap KN \subseteq H\cap L=H'.
\end{equation*}
Therefore, we have
\[H' K \subseteq HK \cap H'KN=H'(HK \cap NK )=H'(H \cap NK)K \subseteq H' H' K=H'K,\]
whence $H'K=HK \cap H'KN$ in $G$. Note that the subset $H'KN$ is $\cC$-clopen in $G$, as $N \lhd_o G$,  and the double coset $HK$ is $\cC$-closed by the assumption (i). Thus $H'K$ is $\cC$-closed as the intersection of closed subsets, so (ii) holds.

Now let us assume (ii) and deduce (i). Then the double coset $HK$ is $\cC$-closed in $G$ because $H \in \cO_\cC(H,G)$ and $H \cap K \subseteq H$. Thus it remains to show that $H \cap K$ is pro-$\cC$ tractable in $G$.

Take any $M \lhd_o G$ and set $L=(H \cap K) M \lo G$. Then $H'=H \cap L \in \cO_\cC(H,G)$ and  we can write
$H= \bigsqcup_{i=1}^n H' h_i$, where $h_1=1$ and $h_i \in H \setminus H'$, for $i=2,\dots,n$. Note that $H \cap K \subseteq H'$, by construction, which easily implies that $h_i \notin H'K$, for $i=2,\dots,n$ (indeed, if $h_i=xy$, where $x \in H'$ and $y \in K$, then $x^{-1}h_i=y \in H \cap K \subseteq H'$, so $h_i \in H'$, whence and $i=1$). By the assumption (ii), the double coset $H'K$ is $\cC$-closed in $G$, hence there exists $N \lhd_o G$ such that
\begin{equation}\label{eq:h_i_notin_H_dash_K}
 h_i \notin H'K N, \text{ for }i=2,\dots,n.
\end{equation}
After replacing $N$ with $N \cap M$, we can suppose that $N \subseteq M$. Let us show that
\[HN \cap KN \subseteq L=(H \cap K)M.\]
Since $HN \cap KN=(H \cap KN)N$ and $N \subseteq L$, it is enough to check that $H \cap KN \subseteq L$. But, in view of \eqref{eq:h_i_notin_H_dash_K}, we know that $H'h_i \cap KN=\emptyset$, for $i=2,\dots,n$, hence $H \cap KN \subseteq H'h_1=H' \subseteq L$, as required. Therefore $H\cap K$ is pro-$\cC$ tractable in $G$  and (i) holds.
\end{proof}

\begin{cor} If $H, K$ are subgroups of a group $G$ then the following are equivalent:
\begin{itemize}
    \item[(i)] the double coset $HK$ is $\cC$-closed  and the intersection $H \cap K$ is pro-$\cC$ tractable in $G$.
    \item[(ii)] for every $H' \in \cO_\cC(H,G)$, with $H \cap K \subseteq H'$, the double coset $H' K$ is $\cC$-closed in $G$;
    \item[(iii)] for every $K' \in \cO_\cC(K,G)$, with $H \cap K \subseteq K'$, the double coset $HK'$ is $\cC$-closed in $G$;
    \item[(iv)] for all $H' \in \cO_\cC(H,G)$ and $K' \in \cO_\cC(K,G)$, with $H \cap K= H' \cap K'$, the double coset $H' K'$ is $\cC$-closed in $G$;
\end{itemize}
\end{cor}

\begin{proof}
The equivalence between (i) and (ii) is the subject of Proposition~\ref{prop:main}, and the equivalence between (i) and (iii) follows by symmetry (or because $HK'=(K'H)^{-1}$). Evidently (iv) implies (ii). Conversely, (iv) follows from (ii) and (iii) because
\[H'K \cap HK'=H'(K \cap HK')=H'(K \cap H)K'=H'K',\]
where the last equality is valid since $K \cap H \subseteq H'$.
\end{proof}

%
%
%
%

\section{A LERF group without the Wilson-Zalesskii property}\label{sec:example}
Throughout this section we assume that $\cC$ is the family of all finite groups. In this case the pro-$\cC$ topology on a group $G$ is the {profinite topology},
$\cC$-open subgroups of $G$ are precisely the finite index subgroups and the $\cC$-closed subsets of $G$ are called {separable}. Recall that $G$ is said to be \emph{ERF} if all subgroups are separable and \emph{LERF} if all finitely generated subgroups are separable.

In this section we show that separability of a double coset $HK$ does not necessarily yield that the intersection $H \cap K$ is profinitely tractable even for finitely generated subgroups $H, K$ of a LERF group $G$. Our construction is based on examples of Grunewald and Segal from \cite{G-S}.

Let $A=M_2(\Z)$ be the additive group of $2 \times 2$ matrices with integer entries, and let $H=SL_2(\Z)$ act on $A$ by left multiplication. We define the group $G$ as the resulting semidirect product $A \rtimes H=M_2(\Z) \rtimes SL_2(\Z)$. Recall that $H$ is finitely generated and virtually free and $A$ is the free abelian group of rank $4$, hence $A$ is ERF and $H$ is LERF, so $G$ is LERF (see \cite[Theorem~4]{All-Greg}).

Denote by $i \in A$ the identity matrix from $M_2(\Z)$ and set $K=iHi^{-1} \leqslant G$.
For any subgroup $F \leqslant H=SL_2(\Z)$ the conjugacy class $i^F=\{fif^{-1} \mid f \in F\}\subseteq A$ is the orbit of the identity matrix under the left action of $F$, so it consists of matrices from $F$, but now considered as a subset of $M_2(\Z)=A$. Since the determinant map $\det: A=M_2(\Z) \to \Z$ is clearly continuous with respect to the profinite topologies on $A$ and $\Z$, the conjugacy class $i^H=\det^{-1}(\{1\})$ is closed in the profinite topology on $A$.

Now let us show that the product $i^H H$ is closed in the profinite topology on $G$. Indeed, suppose that $xy \in G \setminus i^H H$, where $x \in A$ and $y \in H$. Then $x \notin i^H$, so there is $m \in \N$ such that for the finite index characteristic subgroup $A'=M_2(m\Z) \lhd A$ we have $x \notin i^H A'$. The latter implies that $xy \notin i^H A' H$. Since $A'H$ is a finite index subgroup of $G$, we see that $i^H A'H$ is a clopen subset in the profinite topology on $G$ containing $i^HH$ but not containing $xy$. Thus $i^H H$ is indeed profinitely closed in $G$. Note that $i^HH=HiH$, thus the double coset $HK=(HiH)i^{-1}$ is separable in $G$.

As Grunewald and Segal observed in \cite[Section~5]{G-S}, $H$ contains a finite index free subgroup $H'$ (in fact, $|H:H'|=36$) such that the orbit of $i$ under the action of $H'$ is not separable in
the profinite topology on $A$ (equivalently, $H'$ is not closed in the congruence topology on $H=SL_2(\Z)$).

Observe that $H'iH=i^{H'} H$, so $H'iH \cap A=i^{H'}$. Since  $i^{H'}$ is not separable in $A$, it follows that the double cosets $H'iH$ and $H'K=(H'iH)i^{-1}$ cannot be separable in $G$ (this is true because the topology on the subgroup $A$, induced from the profinite topology on $G$, is always weaker than the profinite topology on $A$).

Finally, we note that $H \cap K=\{1\}$ because the $H$-stabiliser of $i$ is trivial, and every finite index subgroup of $H$ belongs to $\cO_{\cC}(H,G)$, as $G$ is LERF and $H$ is finitely generated (see, for example, \cite[Lemma 4.17]{M-M}).
Thus we have constructed the following example.

\begin{ex}
There is a LERF group $G$ (isomorphic to a split extension of $\Z^4$ by $SL_2(\Z)$) and finitely generated subgroups $H,K \leqslant G$ such that $H \cap K=\{1\}$ and the double coset $HK$ is separable in $G$, but the double coset $H'K$ is not separable in $G$, for some finite index subgroup $H' \leqslant_f H$. We deduce, from Proposition~\ref{prop:main}, that the intersection $H \cap K$ is not profinitely tractable in $G$, so $G$ does not have the Wilson-Zalesskii property by Proposition~\ref{prop:tract_intersec_equiv}.
\end{ex}

\end{document}